\documentclass[a4paper,10pt]{article}
\usepackage{amsmath}
\usepackage{amsfonts}
\usepackage{amssymb}
\usepackage{graphicx}
\usepackage{float}
\usepackage{relsize}
\usepackage{tikz}
\usetikzlibrary{arrows}

\usepackage{color}%
\providecommand{\U}[1]{\protect\rule{.1in}{.1in}}
\newtheorem{theorem}{Theorem}

\newtheorem{corollary}[theorem]{Corollary}

\newtheorem{example}[theorem]{Example}

\newtheorem{lemma}[theorem]{Lemma}

\newenvironment{proof}[1][Proof]{\noindent\textbf{#1.} }{\ \rule{0.5em}{0.5em}}

\newenvironment{proof5}[1][Proof of Theorems \ref{main3}]{\noindent\textbf{#1.} }{\ \rule{0.5em}{0.5em}}
\newenvironment{proof6}[1][Proof of Theorem \ref{main7}]{\noindent\textbf{#1.} }{\ \rule{0.5em}{0.5em}}
\newenvironment{corollary1}[1][Conjecture 1]{\noindent\textbf{#1.} }{\ \rule{0.5em}{0.5em}}
\begin{document}

\title{On a conjecture concerning the extensions of a reciprocal matrix}
\author{Ros\'{a}rio Fernandes \thanks{Email: mrff@fct.unl.pt The work of this author
was supported in part by FCT- Funda\c{c}\~{a}o para a Ci\^{e}ncia e
Tecnologia, under project UIDB/00297/2020.}\\CMA and Departamento de Matem\'{a}tica \\Faculdade de Ci\^{e}ncias e Tecnologia \\Universidade Nova de Lisboa\\2829-516 Caparica, Portugal}
\maketitle

\begin{abstract}
Let $A$ be a reciprocal matrix of order $n$ and $w$ be its Perron eigenvector. To infer the efficiency of $w$ for $A$, based on the principle of Pareto optimal decisions, we study the strong connectivity of a certain digraph associated with $A$ and $w$. A reciprocal matrix $B$ of order $n+1$ is an extension of $A$ if the matrix $A$ is obtained from $B$ by removing its last row and column. We prove that there is no extension of a reciprocal matrix whose digraph associated with the extension and its Perron eigenvector has a source, as conjectured by Furtado and Johnson in "Efficiency analysis for the Perron vector of a reciprocal matrix". As an application, considering $n\geq 5$ and $A$ a matrix obtained from a consistent one by perturbing four
entries above the main diagonal, $x,y,z,a$, and the corresponding reciprocal entries, in
a way that there is a submatrix of size $2$ containing the four perturbed
entries and not containing a diagonal entry,
we describe the relations among $x,y,z,a$ with which $A$ always has efficient Perron eigenvector.
\end{abstract}


\textbf{AMS Subject of Classification}: 05C50, 15A18, 15A24, 05C45, 90B50

\section{Introduction}

\hspace{3ex}A matrix $A=[a_{ij}]\in M_{n}$ (the set of $n\times n$ real
matrices) with positive entries is said to be a \emph{reciprocal matrix} or a \emph{pairwise comparison
matrix}, if $a_{ij}%
=1/a_{ji}$ for all $i,j=1,\ldots,n.$ In particular, $a_{ii}=1$ for all
$i=1,\ldots,n.$ We denote by ${\cal PC}_n$ the set of all reciprocal matrices of order $n$. The AHP is a method of pairwise comparision alternatives, \cite{saaty1977,Saaty,thakkar2021multi}, where Saaty proposed the construction of a reciprocal matrix $A$ whose entry $(i,j)$ indicates the strength which alternative $i$ dominates alternative $j$ relatively to a fixed criterion/alternative.

A  matrix $A\in {\cal PC}_n$ is said to be \emph{transitive} or
\emph{consistent} if $a_{ij}a_{jk}=a_{ik}$ for all $i,j,k=1,\ldots,n.$
It is well known that a
consistent matrix $A\in {\cal PC}_n$ can be written as $ww^{-1},$ for some positive
vector $w=\left[
\begin{array}
[c]{ccc}%
w_{1} & \cdots & w_{n}%
\end{array}
\right]  ^{T},$ with
\begin{equation}
w^{-1}=\left[
\begin{array}
[c]{ccc}%
w_{1}^{-1} & \cdots & w_{n}^{-1}%
\end{array}
\right]  . \label{w-1}%
\end{equation}
Any  matrix in ${\cal PC}_2$ is consistent, but for $n>2$ this is not
generally true.

As many times a matrix $A\in {\cal PC}_n$ is not consistent, Saaty proposed a unit positive eigenvector associated with the largest eigenvalue of $A$ to the vector that gives the priorities for the alternatives relatively to the fix criterion/alternative. As we mentioned before, with this vector a consistent matrix for the problem can be constructed. However, as he noted with his consistency ratio, sometimes this vector is not a good choice for the vector of priorities.  So, we need to
reestimate the elements of the matrix $A$ to realize a
better vector or we need to use another method to analyse the vector of priorities. One of this methods is the efficiency of the vector for $A$, see \cite{blanq2006,european} and the references therein.

A positive vector $w=\left[
\begin{array}
[c]{ccc}%
w_{1} & \cdots & w_{n}%
\end{array}
\right]  ^{T}$ is said to be \emph{efficient} for a matrix $A=[a_{ij}]\in
{\cal PC}_{n}$ if there is no other vector $w^{\prime}=\left[
\begin{array}
[c]{ccc}%
w_{1}^{\prime} & \cdots & w_{n}^{\prime}%
\end{array}
\right]  ^{T}$ such that
\[
\left\vert a_{ij}-\frac{w_{i}^{\prime}}{w_{j}^{\prime}}\right\vert
\leq\left\vert a_{ij}-\frac{w_{i}}{w_{j}}\right\vert \text{ \qquad for all
}1\leq i,j\leq n,
\]
with the inequality being strict for at least one pair $(i,j).$ Otherwise, we say that $w$ is an inefficient vector for $A$. Note that if
$w$ is efficient for $A$ then any positive multiple of $w$ is also efficient.

In \cite{blanq2006}, a characterization of the efficiency of a vector $w$ was
presented in terms of a certain digraph associated with $A$ and $w$, which we denote
throughout by $G_{A,w}$ and describe in detail in the next section.

The Perron eigenvector of a matrix $A\in {\cal PC}_n$ is the eigenvector associated with the largest eigenvalue (the Perron eigenvalue) having its first component equal to one. The study of the efficiency of the Perron eigenvector for a reciprocal matrix, obtained from a consistent one by perturbing some entries, has been done only in some special cases, as we will describe in Section \ref{s3}. However, when this vector is inefficient, some authors suggest to construct a reciprocal matrix with Perron eigenvector efficient and having the initial reciprocal matrix as one of its principal submatrices, \cite{f2}.

Let $C\in {\cal PC}_{n+1}$ and $i$ be an integer such that $1\leq i\leq n+1$. We denote by $C(i)$ the matrix obtained from $C$ by removing row and column $i$. An {\it extension} of a ${\cal PC}_n$ matrix $A$ is a matrix $B\in {\cal PC}_{n+1}$ such that $B(n+1)=A$.

In \cite{f2}, the authors wrote the following conjecture:
\vspace{1ex}

\hspace{-3ex}{\bf Conjecture 1.} ``There is no extension $B$ of a reciprocal matrix $A$ such that the digraph $G_{B,w}$ has a source, where $w$ is the Perron eigenvector of $B$."
\vspace{1ex}

In this paper, we prove this conjecture. As a consequence of this result, we
study the efficiency of the Perron
eigenvector of  $A\in {\cal PC}_n$, with $n\geq 5$ and where $A$ is obtained from a consistent matrix by
perturbing four specific entries and their reciprocals. More precisely, we consider the case in which four
entries and their reciprocals are perturbed, as long as there is a $2\times2$
submatrix containing  the four perturbed entries and not containing a
principal entry of $A$.

Since the efficiency of a positive multiple of the Perron eigenvector for a reciprocal matrix $A$ can be analyzed
considering a permutation similarity of $A$ followed by a similarity by a
diagonal matrix with positive diagonal, \cite{Fer},
we will mainly focus on $n\times n$ reciprocal matrices of the form
\begin{equation}
Z_{n}(x,y,z,a):=\left[
\begin{array}
[c]{cccccccc}%
1 & 1 & 1 & \cdots & \cdots & 1 & y & x\\
1 & 1 & 1 & \cdots & \cdots & 1 & a & z\\
1 & 1 & 1 & \cdots & \cdots & 1 & 1 & 1\\
\vdots & \vdots & \vdots & \ddots & \ddots & \vdots & \vdots & \vdots\\
\vdots & \vdots & \vdots & \ddots & \ddots & \vdots & \vdots & \vdots\\
1 & 1 & 1 & \cdots & \cdots & 1 & 1 & 1\\
\smallskip\frac{1}{y} & \smallskip\frac{1}{a} & 1 & \cdots & \cdots & 1 & 1 & 1\\
\frac{1}{x} & \frac{1}{z} & 1 & \cdots & \cdots & 1 & 1 & 1
\end{array}
\right]  ,\label{recmatrix}%
\end{equation}
with $x,y,z,a>0.$

Reciprocal matrices as (\ref{recmatrix}) can occur in real life. Using a similar example as in \cite{Fer2}, consider some criterion $\cal C$, like the weight or the price, and $5$ alternatives, for example $5$ different chairs, $x_1,x_2,x_3,x_4,x_5$, such that $${\cal C}(x_1)>{\cal C}(x_2)>{\cal C}(x_3)>{\cal C}(x_4)>{\cal C}(x_5), $$ where ${\cal C}(x_i)$ denotes the value of the alternative $x_i$ for the criterion ${\cal C}$, with $i=1,2,3,4,5$. If the numbers ${\cal C}(x_1)$, with $i=1,2,3,4,5$, are so close that we only consider the pairs $$({\cal C}(x_1),{\cal C}(x_4)),\ ({\cal C}(x_1),{\cal C}(x_5)),\ ({\cal C}(x_2),{\cal C}(x_4)),\ ({\cal C}(x_2),{\cal C}(x_5))$$ and their inverses having different coordinates then the reciprocal matrix is the matrix (\ref{recmatrix}), with $n=5$.

\vspace{2ex}

The paper is organized as follows. In Section \ref{sback}, we describe in detail the digraph associated with a reciprocal matrix and a vector. When the vector is the Perron eigenvector, we prove the main result of this paper, about certain kind of vertices there are not in this digraph.   In Section \ref{s1}, we introduce known results in the literature related with the extensions of a ${\cal PC}_n$ matrix, which we also comment in on detail. Moreover, we prove the Conjecture 1. In Section \ref{s3}, after to describe known results in the literature related with the efficiency of the Perron eigenvector of some ${\cal PC}_n$  matrices obtained from a consistent matrix by perturbing some entries, we study a generalization of these results,  the matrix (\ref{recmatrix}).
In Section \ref{sproof}, we prove the results presented in Section \ref{s3}. We conclude
this paper with some final remarks and an open question in Section \ref{scon}.
\vspace{1ex}

The calculations in this paper were made with the help of the software Maple
and Octave.

\section{Digraph associated with a ${\cal PC}_n$ matrix \label{sback}}

\hspace{3ex}From Perron-Frobenius Theorem for positive matrices \cite{HJ}, it is known that the Perron eigenvalue $r$ of a matrix $A=[a_{ij}]\in {\cal PC}_{n}$ is
a simple positive eigenvalue and any
eigenvector associated  with it has nonzero entries with constant sign.  Moreover, $r\geq n$.

In \cite{blanq2006}, a characterization of the efficiency of a vector $w=\left[
\begin{array}
[c]{ccc}%
w_{1} & \cdots & w_{n}%
\end{array}
\right]  ^{T}$ for a reciprocal matrix $A\in M_{n}$ was presented using the following digraph $G_{A,w}$ associated with $A$ and $w$: the vertex set of $G_{A,w}=(V,E)$ is $V=\{1,\ldots,n\}$ and
the edge set is%
\[
E=\left\{(i,j):\frac{w_{i}}{w_{j}}\geq a_{ij}\text{, }i\neq j\right\}.
\]
It follows that if $\frac{w_i}{w_j}=a_{ij}$ then the edges $(i,j)$ and $(j,i)$ are in $G_{A,w}$.

Throughout this paper, sometimes we denote by  $i\rightarrow j$  the edge $(i,j)$.

In \cite{blanq2006} the following result was given.

\begin{theorem}
\label{blanq} Let $A\in {\cal PC}_{n}$. A positive vector $w$ is
efficient for $A$ if and only if $G_{A,w}$ is a strongly connected digraph.
\end{theorem}

\vspace{1ex}

As noted in \cite{Rub}, the strong connectivity of $G_{A,w}$ is equivalent to the existence of a directed hamiltonian cycle in $G_{A,w}$.

The next result shows a restriction on the vertices of $G_{A,w}$ when $w$ is the Perron eigenvector of $A$.

\begin{theorem}\label{t1} Let $A\in {\cal PC}_n$, with $n\geq 3$. Let $w$ and $G_{A,w}=(V,E)$ be its Perron eigenvector and its digraph associated with $w$. Let $i\in V$ and $k\in\{1,\ldots ,n\}\setminus\{i\}$ such that $(k,i)\not\in E$. Then, there is $j\in V$ such that $(j,i)\in E$ and $(i,j)\not\in E$.
\end{theorem}

\begin{proof} Let $w=\left[
\begin{array}
[c]{ccc}%
w_{1} & \cdots & w_{n}%
\end{array}
\right]  ^{T}$ and $A=[a_{ij}]$. Let $r$ be the Perron eigenvalue of $A$. Suppose that  there is not $j\in V$ such that $(j,i)\in E$ and $(i,j)\not\in E$. Then, $$\frac{w_i}{w_j}\geq a_{ij},\ \ \ \forall j\in\{1,\ldots ,n\}\setminus\{i,k\}$$ and $$\frac{w_i}{w_k}> a_{ik}.$$ Consequently, $w_i\geq a_{ij}w_j$, for all $j\in\{1,\ldots ,n\}\setminus\{i,k\},$ and   $w_i\geq a_{ik}w_k$. From the $i$th row of $Aw=rw$ we get $$\sum_{\ell =1}^n a_{i \ell}w_\ell=rw_i.$$ So, $$nw_i>\sum_{j=1, j\neq k}^n a_{i j}w_j+a_{i k}w_k=rw_i$$ and $n>r$, which is impossible.
\end{proof}

\vspace{2ex}

\begin{corollary}\label{t2} Let $A\in {\cal PC}_n$, with $n\geq 3$, and $w$ be its Perron eigenvector. Then, $G_{A,w}$ has no source vertex.
\end{corollary}

\vspace{2ex}

If $w$ is not the Perron eigenvector of $A$ then the digraph $G_{A,w}$ may have a source. For example, if $A=\left[\begin{array}{ccc}1&1&2\\ 1&1&1\\ \frac{1}{2}&1&1\end{array}\right]$ and $w=\left[\begin{array}{ccc} 1&2&3\end{array}\right]^T$ then the vertex $3$ is a source of $G_{A,w}$.

\section{Extensions of a ${\cal PC}_n$ matrix \label{s1}}

\hspace{3ex}In \cite{Fer2}, we studied a particular extension of the matrix (\ref{recmatrix}), adding a row and column with all entries equal to one. Beginning with a reciprocal matrix with an inefficient Perron eigenvector such that its extension has an efficient Perron eigenvector, we observed that the relations of the weights of the corresponding components of these two Perron eigenvectors may or may not change (Table 4 of \cite {Fer2}).

If $A\in {\cal PC}_n$ is not consistent, Furtado and Johnson, in \cite{f2}, gave an algorithm to construct an extension of $A$ with inefficient Perron eigenvector (section 4.2 of \cite{f2}). At the same paper, Theorem 25 shows that if $A$ is consistent then all extensions of $A$ will have efficient Perron eigenvectors. Moreover, when $A$ is not consistent, they proved that there is an extension of $A$ with efficient Perron eigenvector (Theorem 33 of \cite{f2}). However, they did not mention any think about the relations of the weights of the corresponding components of the Perron eigenvectors of $A$ and its extensions.

Using the construction of an extension exhibits in \cite{f2}, we conclude that sometimes the extension of $A$ with efficient Perron eigenvector do not give us a good vector of priorities for $A$. Example \ref{ee1} shows us this problem.

\vspace{2ex}

For $1\leq i\leq n$, let $r_i$ be the sum of the entries in row $i$ of $A$. The reciprocal matrix $A$ is said {\it well-behaved of type I} if $r_1-r_n>1$ (see \cite{f2}).

\begin{theorem} \cite {f2} \label{t-12} Let $A\in {\cal PC}_n$ and suppose that $e_{n-1}$ (vector with $n-1$ coordinates equal to one) is efficient for $A(n)$ and the Perron eigenvector of $A$ is inefficient for $A$. Then, $A(n)$ is not well-behaved.
\end{theorem}

\begin{example} \label{ee1} Let $B$ be the following reciprocal matrix $$B=\left[\begin{array}{ccccc}1&1&1&0.9933&2.5\\ 1&1&1&0.6666&1\\  1&1&1&0.6666&0.5\\ 1.0067&1.5&1.5&1&0.75\\ 0.4&1&2&1.3333&1\end{array}\right].$$ The Perron eigenvector of $B$ is $$W=\left[\begin{array}{ccccc}1&0.7110&0.6325&0.8555&0.8258\end{array}\right]^T=
\left[\begin{array}{ccccc}w_1&w_2&w_3&w_4&w_5\end{array}\right]^T.$$ Thus, $$w_1>w_4>w_5>w_2>w_3.$$

If $D=\left[\begin{array}{ccccc}\frac{1}{2}&0&0&0&0\\ 0&\frac{1}{2}&0&0&0\\  0&0&\frac{1}{2}&0&0\\ 0&0&0&\frac{1}{3}&0\\ 0&0&0&0&\frac{1}{4}\end{array}\right]$ then $$B'=DBD^{-1}=\left[\begin{array}{ccccc}1&1&1&1.49&5\\ 1&1&1&1&2\\  1&1&1&1&1\\ 1/1.49&1&1&1&1\\ 1/5&1/2&1&1&1\end{array}\right].$$ The matrix $B'$ was analysed in \cite{Fer2}, Table 3, and its Perron eigenvector is inefficient. However, the vector $e_5=\left[\begin{array}{ccccc}1&1&1&1&1\end{array}\right]^T$ is efficient for $B'$ (note that all entries in the third row of $B'$ are equal to one). Moreover, $B'$ is well-behaved.

We construct an extension $A'$ of $B'$ with constant row sums. So, the vector $e_6=\left[\begin{array}{cccccc}1&1&1&1&1&1\end{array}\right]^T$ is the Perron eigenvector of $A'$ and, by Theorem \ref{t-12}, is efficient for $A'$. Moreover, $A=(D^{-1}\oplus [1])A'(D\oplus [1])$ satisfies $A(6)=B$ and the Perron eigenvector of $A$, $$V=\left[\begin{array}{cccccc}1&1&1&3/2&2&1/2\end{array}\right]^T=
\left[\begin{array}{cccccc}v_1&v_2&v_3&v_4&v_5&v_6\end{array}\right]^T,$$ is efficient for $A$. Thus, $$v_5>v_4>v_3=v_2=v_1.$$ Therefore, the relations of the weights of the corresponding components of the Perron eigenvectors of $B$ and its extension $A$ change.

\vspace{2ex}

If we consider $B'$ our initial matrix, then its Perron eigenvector is the same as the Perron eigenvector of $B$ and is not efficient. Since $A'$ is the extension of $B'$ with Perron eigenvector $e_6=\left[\begin{array}{cccccc}1&1&1&1&1&1\end{array}\right]^T$, we get a vector with all components equal.
\end{example}

\vspace{2ex}

Using Corollary \ref{t2} we get the Conjecture 1.
\vspace{1ex}

\begin{corollary1} There is no extension $B$ of a ${\cal PC}_n$ matrix $A$ for which the associated digraph of $B$ and its Perron eigenvector has a source vertex.
\end{corollary1}

\section{Efficiency of the Perron eigenvector\label{s3}}

\hspace{3ex}Let $A,B\in {\cal PC}_{n}$. $A$ is {\it monomial similar} to $B$ if $B=QAQ^{-1}$, where $Q$ is a product of a diagonal matrix with positive diagonal entries and a permutation matrix, \cite{Fer}.

Fernandes and Furtado, in \cite{Fer}, proved the next lemma.

\begin{lemma} \label{lem1} Let $A,B\in {\cal PC}_{n}$ monomial similar with $B=QAQ^{-1}$, where $Q$ is a product of a diagonal matrix with positive diagonal entries and a permutation matrix. Let $w$ be the Perron eigenvector of $A$. Then
\begin{enumerate}
\item the vector $Qw$ is a positive multiple of the Perron eigenvector of $B$.
\item The digraphs $G_{A,w}$ and $G_{B,Qw}$ are isomorphic.
\end{enumerate}
\end{lemma}

If $A$ is obtained from a consistent matrix by perturbing one entry and its reciprocal then $A$ is monomial similar to  $Z_n(x,1,1,1),$ with $x>0$ and $n\geq 3$.
  \'{A}bele-Nagy and Boz\'{o}ki, in \cite{p6}, proved that all matrices $Z_n(x,1,1,1)$, with $x>0$, have efficient Perron eigenvectors.
  \vspace{1ex}

  On the other hand, if $A$ is obtained from a consistent matrix by perturbing two entries and their reciprocals then $A$ is monomial similar to $Z_n(x,y,1,1)$, with $x,y>0$ and $n\geq 3$, or to $Z_n(x,1,1,a)$, with $x,a>0$ and $n\geq 4$.
  \'{A}bele-Nagy and et., in \cite{p2}, proved that all matrices $Z_n(x,y,1,1)$ and $Z_n(x,1,1,a)$, with $x,y,a>0$, have efficient Perron eigenvectors.

   Continuing this study, in \cite{Fer}, we studied the matrix $Z_n(x,y,z,1)$, with $x,y,z>0$ and $n\geq 4$.

Note that, the matrix studied in \cite{f1}, $$A=\left[
\begin{array}
[c]{cccccc}%
1 & a & b & 1 & \cdots & 1\\
\smallskip\frac{1}{a} & 1 & c & 1 & \cdots & 1 \\
\frac{1}{b} & \frac{1}{c} & 1 & 1 & \cdots & 1 \\
\vdots & \vdots & \vdots & \vdots & \ddots & \vdots \\
1 & 1 & 1 & 1 & \cdots & 1
\end{array}
\right]\in {\cal PC}_{n},$$ with $a,b,c>0$, when $n= 4$ is monomial similar to $Z_4\left(\frac{b}{a},\frac{1}{a},c,1\right)$.

If $w$ is the Perron eigenvector and $r$ is the Perron eigenvalue of $Z_n(x,y,z,1)$, using the characteristic equation, $(Z_n(x,y,z,1)-rI_n)w=0$, Fernandes and Furtado, in \cite{Fer}, found some conditions for the efficiency of $w$ in $Z_n(x,y,z,1)$. Later, using the same technique, Furtado and Johnson, in \cite{f1}, proved the next theorem.

\begin{theorem}\label{fujo} Let $n= 4$ and $Z_4(x,y,z,1)$ be the matrix in (\ref{recmatrix}), with $x,y,z>0$. If one of the following conditions hold:
\begin{enumerate}
\item $y\leq x\leq z$, $y\leq 1$, $1\leq z$;
\item $y\leq x$, $y\leq 1$, $z\leq 1$, $z\leq x$;
\item $1\leq y\leq x$, $1\leq z\leq x$;
\item $z\leq x\leq y$, $1\leq y$, $z\leq 1$;
\item $x\leq y$, $1\leq y$, $1\leq z$, $x\leq z$;
\item $x\leq y\leq 1$, $x\leq z\leq 1$,
\end{enumerate}
then the Perron eigenvector of $Z_4(x,y,z,1)$ is efficient.

\end{theorem}

On the other hand, observing that the efficiency of the Perron eigenvector of $Z_n(x,y,z,1)$ is connected with the relations among $1,x,y,z$, Fernandes and Palheira, in \cite{Fer2}, proved the two following theorems. Note that, in the versions that appear in \cite{Fer2}, we did not mention the boundary conditions, however in their proofs, we studied these cases.

\begin{theorem} \label{maint1}
\label{main}Let $n= 4$ and $Z_{4}(x,y,z,1)$ be the matrix in (\ref{recmatrix}%
), with $x,y,z>0.$ If neither
\begin{enumerate} \item[(i)] $\min\{1,x\} < \min\{y,z\}< \max\{1,x\} < \max\{y,z\},$ $1\neq x$ and $y\neq z$, nor \item[(ii)] $\min\{y,z\} < \min\{1,x\}< \max\{y,z\} < \max\{1,x\},$ $1\neq x$ and $y\neq z$\end{enumerate} then the Perron eigenvector of $Z_{4}(x,y,z,1)$ is efficient.
\end{theorem}

\begin{theorem}
\label{main2}Let $n\geq 5$ and $Z_{n}(x,y,z,1)$ be the matrix in (\ref{recmatrix}%
), with $x,y,z>0.$ If neither
\begin{enumerate} \item[(i)] $1 < z< x < y,$ nor \item[(ii)] $z < 1< y < x,$ nor \item[(iii)] $z< x< y < 1,$ nor \item[(iv)] $x < z< 1 < y,$ \end{enumerate} then the Perron eigenvector of $Z_{n}(x,y,z,1)$ is efficient.
\end{theorem}

Note that for $n=4$, Theorem \ref{fujo} is equivalent to Theorem \ref{main}. However, for $n\geq 5$, in Theorem \ref{main2} in addition to the cases mentioned in Theorem \ref{main} for the efficiency of the Perron eigenvector of $Z_n(x,y,z,1)$, we also mention the cases $\max\{x,y\}\leq \min\{z,1\}$ and $\max\{1,y\}\leq \min\{z,x\}$ that provide efficient Perron eigenvectors for $Z_n(x,y,z,1)$.  In \cite{f4} the authors used the relations among $1,x,y,z$, shown in Theorem \ref{main2}, to prove that all positive linear combinations of the columns of $Z_n(x,y,z,1)$ are efficient vectors for $Z_n(x,y,z,1)$, with $n\geq 5$.

\vspace{2ex}

Using Theorem \ref{t1}, we will study the efficiency of the Perron eigenvector of $Z_n(x,y,z,a)$, with $n\geq 5$ and  $x,y,z,a>0$.

\begin{theorem}
\label{main3}Let $n\geq 5$ and $Z_{n}(x,y,z,a)$ be the matrix in (\ref{recmatrix}%
), with $x,y,z,a>0.$ Suppose that $x\leq\min\{a,y,z\}$. If neither
\begin{enumerate} \item[(i)] $x<z<a<y,\ z<1$ nor \item[(ii)]  $x < y< a < z,\ 1<a$ nor  \item[(iii)]  $x\leq a< 1< \min\{y,z\}$,\end{enumerate} then the Perron eigenvector of $Z_{n}(x,y,z,a)$ is efficient.
\end{theorem}

Using Lemma \ref{lem1} and the fact that $Z_{n}(x,y,z,a)$ is monomial similar to $Z_{n}(y,x,a,z)$, we get the next theorem.

\begin{theorem}
\label{main4}Let $n\geq 5$ and $Z_{n}(x,y,z,a)$ be the matrix in (\ref{recmatrix}%
), with $x,y,z,a>0.$ Suppose that $y\leq\min\{a,y,z\}$. If neither
\begin{enumerate} \item[(i)] $y<a<z<x,\ a<1$ nor \item[(ii)]  $y < x< z < a,\ 1<z$ nor  \item[(iii)]  $y\leq z< 1< \min\{a,x\}$,\end{enumerate} then the Perron eigenvector of $Z_{n}(x,y,z,a)$ is efficient.
\end{theorem}

Using Lemma \ref{lem1} and the fact that $Z_{n}(x,y,z,a)$ is monomial similar to $Z_{n}(z,a,x,y)$, we get the next theorem.

\begin{theorem}
\label{main5}Let $n\geq 5$ and $Z_{n}(x,y,z,a)$ be the matrix in (\ref{recmatrix}%
), with $x,y,z,a>0.$ Suppose that $z\leq\min\{a,y,x\}$. If neither
\begin{enumerate} \item[(i)] $z<x<y<a,\ x<1$ nor \item[(ii)]  $z < a< y < x,\ 1<y$ nor  \item[(iii)]  $z\leq y< 1< \min\{a,x\}$,\end{enumerate} then the Perron eigenvector of $Z_{n}(x,y,z,a)$ is efficient.
\end{theorem}

Using Lemma \ref{lem1} and the fact that $Z_{n}(x,y,z,a)$ is monomial similar to $Z_{n}(a,z,y,x)$, we get the next theorem.

\begin{theorem}
\label{main6}Let $n\geq 5$ and $Z_{n}(x,y,z,a)$ be the matrix in (\ref{recmatrix}%
), with $x,y,z,a>0.$ Suppose that $a\leq\min\{a,y,x\}$. If neither
\begin{enumerate} \item[(i)] $a<y<x<z,\ y<1$ nor \item[(ii)]  $a < z< x < y,\ 1<x$ nor  \item[(iii)] $ a\leq x< 1< \min\{y,z\}$,\end{enumerate} then the Perron eigenvector of $Z_{n}(x,y,z,a)$ is efficient.
\end{theorem}

Note that when  $a=1$, if neither
\begin{itemize}
\item  $x<z<1<y$ (by Theorem \ref{main3}) nor
\item $z<x<y<1$ (by Theorem \ref{main5}) nor
\item $z < 1< y < x$ (by Theorem \ref{main5}) nor
\item $1 < z< x < y$ (by Theorem \ref{main6})
\end{itemize}
then the Perron eigenvector of $Z_{n}(x,y,z,1)$ is efficient. So, we get Theorem \ref{main2}.

\vspace{1ex}

If $y=z=1$ then there are no restrictions in the last four theorems. Consequently, $Z_n(x,1,1,a)$ has efficient Perron eigenvector, for all $x,a>0$, as mentioned by \'{A}bele-Nagy and et., in \cite{p2}.

\vspace{1ex}

As in \cite{Fer2}, we finish this section with a result about the characterization of the digraph $G_{Z_n(x,y,z,a),w}$, when $w$ is the Perron eigenvector of $Z_n(x,y,z,a)$ and is inefficient.

\begin{theorem}
\label{main7}Let $n\geq 5$, $Z_{n}(x,y,z,a)$ be the matrix in (\ref{recmatrix}%
), with $x,y,z,a>0$, and $w$ be the  Perron eigenvector of $Z_{n}(x,y,z,a)$. Then the vector $w$ is inefficient for $Z_{n}(x,y,z,a)$ if and only if the digraph $G_{Z_n(x,y,z,a),w}$ has a sink.
\end{theorem}

\section{Proof of the results of Section \ref{s3}\label{sproof}}

\hspace{3ex}We next give some lemmas from which our results easily follow.
Suppose that $n\geq 5.$  By $r$ we denote the Perron
eigenvalue of the matrix $Z_{n}(x,y,z,a),$ with $x,y,z,a>0.$ Recall
 that $r\geq n.$ We denote by
\[
w=\left[
\begin{array}
[c]{ccccc}%
w_{1} & w_{2} &\cdots & w_{n-1} & w_{n}%
\end{array}
\right]  ^{T}
\]
the Perron vector of $Z_{n}(x,y,z,a)$
associated with $r$.
Since
$(Z_{n}(x,y,z,a)-rI_{n})w=0,$ a calculation shows that%
\begin{align}
 (1-r)w_{1}+w_{2}+\sum_{i=3}^{n-2}w_i+yw_{n-1}+xw_{n}
&=0\label{j1}\\
w_{1}+(1-r)w_{2}+\sum_{i=3}^{n-2}w_i+aw_{n-1}+zw_{n}
&=0\label{j2}\\
 w_{1}+w_{2}+\sum_{i=3,i\neq j}^{n-2}w_i+(1-r)w_{j}+w_{n-1}+w_{n} & =0\label{j3}\\
\frac{1}{y}w_{1}+\frac{1}{a}w_{2}+\sum_{i=3}^{n-2}w_i+(1-r)w_{n-1}+w_{n} &
=0\label{j4}\\
\frac{1}{x}w_{1}+\frac{1}{z}w_{2}+\sum_{i=3}^{n-2}w_i+w_{n-1}+(1-r)w_{n}  &
=0, \label{j5}%
\end{align}
in which,  equation (\ref{j3}) is for $3\leq j\leq n-2$.

If $n>5$, subtracting both hand-sides of equation (\ref{j3}) with $j=3$ from those of  equation (\ref{j3}) with $j\neq 3$ and $4\leq j\leq n-2$, we get $$w_j=w_3.$$

Using last equality, in the next equations the expression $(\ell)-s(j)$ means the result of subtracting both hand-sides of equation $(\ell)$ from those of equation
$(j)$, multiplied by $s$.
{\footnotesize \begin{align}
(\ref{j1})-(\ref{j2})\hspace{5ex}& r(w_2-w_1)+(y-a)w_{n-1}+(x-z)w_n=0\label{j13}\\
(\ref{j1})-(\ref{j3})\hspace{5ex}& r(w_3-w_1)+(y-1)w_{n-1}+(x-1)w_n=0\label{j14}\\
(\ref{j1})-y(\ref{j4})\hspace{5ex}& r(yw_{n-1}-w_1)+(1-\frac{y}{a})w_{2}+(1-y)(n-4)w_3+(x-y)w_n=0\label{j15}\\
(\ref{j1})-x(\ref{j5})\hspace{5ex}& r(xw_{n}-w_1)+\left(1-\frac{x}{z}\right)w_{2}+(1-x)((n-4)w_3)+(y-x)w_{n-1}=0\label{j16}\\
(\ref{j2})-(\ref{j3})\hspace{5ex}& r(w_3-w_2)+(a-1)w_{n-1}+(z-1)w_{n}=0\label{j17}\\
(\ref{j2})-a(\ref{j4})\hspace{5ex}& r(aw_{n-1}-w_2)+\left(1-\frac{a}{y}\right)w_{1}+(1-a)(n-4)w_3+(z-a)w_n=0\label{j18}\\
(\ref{j2})-z(\ref{j5})\hspace{5ex}& r(zw_n-w_2)+\left(1-\frac{z}{x}\right)w_{1}+(1-z)(n-4)w_3+(a-z)w_{n-1}=0\label{j19}\\
(\ref{j3})-(\ref{j4})\hspace{5ex}& r(w_{n-1}-w_3)+\left(1-\frac{1}{y}\right)w_{1}+\left(1-\frac{1}{a}\right)w_{2}=0\label{j20}\\
(\ref{j3})-(\ref{j5})\hspace{5ex}& r(w_{n}-w_3)+\left(1-\frac{1}{x}\right)w_{1}+\left(1-\frac{1}{z}\right)w_2=0\label{j21}\\
(\ref{j4})-(\ref{j5})\hspace{5ex}& r(w_{n}-w_{n-1})+\left(\frac{1}{y}-\frac{1}{x}\right)w_{1}+\left(\frac{1}{a}-\frac{1}{z}\right)w_2=0.\label{j22}%
\end{align}}

We next give same lemmas from which our results easily follow.

\begin{lemma} \label{le-1} Let $n\geq 5$ and $Z_n(x,y,z,a)$, with $x,y,z,a>0$, be the matrix in (\ref{recmatrix}). Let $w$ be the Perron vector of $Z_n(x,y,z,a)$.
\begin{enumerate}
\item If $a\leq y$ and $z\leq x$ then the edge $(1,2)$ is in $G_{Z_n(x,y,z,a),w}$.
\item If $y\leq\min\{1,a,x\}$ then the edge $(1,n-1)$ is in $G_{Z_n(x,y,z,a),w}$.
\item If $x\leq\min\{1,y,z\}$ then the edge $(1,n)$ is in $G_{Z_n(x,y,z,a),w}$.
\item If $a\leq\min\{1,y,z\}$ then the edge $(2,n-1)$ is in $G_{Z_n(x,y,z,a),w}$.
\item If $z\leq\min\{1,x,a\}$ then the edge $(2,n)$ is in $G_{Z_n(x,y,z,a),w}$.
\item If $y\leq x$ and $a\leq z$ then the edge $(n-1,n)$ is in $G_{Z_n(x,y,z,a),w}$.
\item If $1\leq \min\{x,y\}$  then, for $3\leq i\leq n-2$, the edge $(1,i)$ is in $G_{Z_n(x,y,z,a),w}$.
\item If $1\leq \min\{a,z\}$  then, for $3\leq i\leq n-2$, the edge $(2,i)$ is in $G_{Z_n(x,y,z,a),w}$.
\item If $\max\{y,a\}\leq 1$  then, for $3\leq i\leq n-2$, the edge $(n-1,i)$ is in $G_{Z_n(x,y,z,a),w}$.
\item If $\max\{x,z\}\leq 1$  then, for $3\leq i\leq n-2$, the edge $(n,i)$ is in $G_{Z_n(x,y,z,a),w}$.
\end{enumerate}
\end{lemma}

\begin{proof}  Let
\[
w=\left[
\begin{array}
[c]{ccccc}%
w_{1} & w_{2} &\cdots & w_{n-1} & w_{n}%
\end{array}
\right]  ^{T}
\]
and $r$ be the Perron eigenvalue of $Z_n(x,y,z,a)$.
Recall that $r>0$, $w$ is a positive eigenvector and $w_i=w_3$, for $3\leq i\leq n-2$.
\begin{enumerate}
\item Since $a\leq y$ and $z\leq x$, using equation (\ref{j13}) we conclude that $w_2-w_1\leq 0.$ This implies $\frac{w_1}{w_2}\geq 1$ and consequently we get the result.
\item If $y\leq\min\{1,a,x\}$, using equation (\ref{j15}) we conclude that $yw_{n-1}-w_1\leq 0.$ This implies $\frac{w_1}{w_{n-1}}\geq y$ and consequently we get the result.
\item If $x\leq\min\{1,y,z\}$, using equation (\ref{j16}) we conclude that $xw_{n}-w_1\leq 0.$ This implies $\frac{w_1}{w_{n}}\geq x$ and consequently we get the result.
\item Since $a\leq\min\{1,y,z\}$, using equation (\ref{j18}) we conclude that $aw_{n-1}-w_2\leq 0.$ This implies $\frac{w_2}{w_{n-1}}\geq a$ and consequently we get the result.
\item Since $z\leq\min\{1,x,a\}$, using equation (\ref{j19}) we conclude that $zw_{n}-w_2\leq 0.$ This implies $\frac{w_2}{w_{n}}\geq z$ and consequently we get the result.
\item Since $y\leq x$ and $a\leq z$, using equation (\ref{j22}) we conclude that $w_{n}-w_{n-1}\leq 0.$ This implies $\frac{w_{n-1}}{w_{n}}\geq 1$ and consequently we get the result.
\item Since $1\leq \min\{x,y\}$, using equation (\ref{j14}) we conclude that $w_3-w_1\leq 0.$ This implies $\frac{w_1}{w_3}\geq 1$ and consequently we get the result.
\item Since $1\leq \min\{a,z\}$, using equation (\ref{j17}) we conclude that $w_{3}-w_2\leq 0.$ This implies $\frac{w_2}{w_{3}}\geq 1$ and consequently we get the result.
\item Since $\max\{y,a\}\leq 1$, using equation (\ref{j20}) we conclude that $w_{3}-w_{n-1}\leq 0.$ This implies $\frac{w_{n-1}}{w_{3}}\geq 1$ and consequently we get the result.
\item Since $\max\{x,z\}\leq 1$, using equation (\ref{j21}) we conclude that $w_{3}-w_{n}\leq 0.$ This implies $\frac{w_{n}}{w_{3}}\geq 1$ and consequently we get the result.
\end{enumerate}
\end{proof}

With an analogous proof we have the next lemma.

\begin{lemma} \label{le-2} Let $n\geq 5$ and $Z_n(x,y,z,a)$, with $x,y,z,a>0$,  be the matrix in (\ref{recmatrix}). Let $w$ be the Perron eigenvector of $Z_n(x,y,z,a)$.
\begin{enumerate}
\item If $y\leq a$ and $x\leq z$ then the edge $(2,1)$ is in $G_{Z_n(x,y,z,a),w}$.
\item If $\max\{1,a,x\}\leq y$ then the edge $(n-1,1)$ is in $G_{Z_n(x,y,z,a),w}$.
\item If $\max\{1,y,z\}\leq x$ then the edge $(n,1)$ is in $G_{Z_n(x,y,z,a),w}$.
\item If $\max\{1,y,z\}\leq a$ then the edge $(n-1,2)$ is in $G_{Z_n(x,y,z,a),w}$.
\item If $\max\{1,a,x\}\leq z$ then the edge $(n,2)$ is in $G_{Z_n(x,y,z,a),w}$.
\item If $x\leq y$ and $z\leq a$ then the edge $(n,n-1)$ is in $G_{Z_n(x,y,z,a),w}$.
\item If $ \max\{x,y\}\leq 1$  then, for $3\leq i\leq n-2$, the edge $(i,1)$ is in $G_{Z_n(x,y,z,a),w}$.
\item If $ \max\{a,z\}\leq 1$  then, for $3\leq i\leq n-2$, the edge $(i,2)$ is in $G_{Z_n(x,y,z,a),w}$.
\item If $1\leq \min\{a,y\}$  then, for $3\leq i\leq n-2$, the edge $(i,n-1)$ is in $G_{Z_n(x,y,z,a),w}$.
\item If $1\leq \min\{x,z\}$  then, for $3\leq i\leq n-2$, the edge $(i,n)$ is in $G_{Z_n(x,y,z,a),w}$.
\end{enumerate}
\end{lemma}

\medskip

We only give the proof of Theorem \ref{main3} because Theorems \ref{main4}, \ref{main5} and \ref{main6} are consequences of Theorem \ref{main3} and Lemma \ref{lem1}.

\begin{proof5} Let
\[
w=\left[
\begin{array}
[c]{cccc}%
w_{1} & w_{2} & w_{3} & w_{4}%
\end{array}
\right]  ^{T}
\]
be the Perron vector  of $Z_n(x,y,z,a)$, associated with the Perron eigenvalue $r$.
Recall that if $n>5$ then $w_i=w_3$, for $4\leq i\leq n-2$, and  $3\rightarrow i$ and $i\rightarrow 3$ are edges of $G_{Z_n(x,y,z,a),w}$. Therefore, if we find  a directed Hamiltonian cycle among the vertices $1,2,3,n-1,n$ of $G_{Z_n(x,y,z,a),w}$ then we get a directed Hamiltonian cycle in $G_{Z_n(x,y,z,a),w}$.

The cases described in the statement of the theorem are:
\begin{itemize}
\item $x<z<1\leq a<y$,
\item $x<z<a\leq 1<y$,
\item $x<z<a<y\leq 1$,
\item $1\leq x<y<a<z$,
\item $x<1\leq y<a<z$,
\item $x<y\leq 1<a<z$,
\item $x\leq a<1<z\leq y$,
\item $x\leq a<1<y\leq z$.
\end{itemize}

To facilitate the proof and using Lemmas \ref{le-1} and \ref{le-2}, in the following tables we show the efficiency of the Perron eigenvector for all cases not described in the statement of the theorem.

In Table \ref{tab1}, we consider the cases that, using Lemmas \ref{le-1} and \ref{le-2}, produce a directed Hamiltonian cycle. If $b,c$ are two real numbers, we write $b,c$ to denote $\min\{b,c\}\leq \max\{b,c\}$.

\begin{table}[H] \centering
\footnotesize \begin{tabular}{|c|c||c|c|} \hline & & & \\ relation among $1,x,y,z,a$ &  cycle & relation among $1,x,y,z,a$ &  cycle \\ \hline\hline  & & & \\
$x\leq 1\leq a\leq y,z$ & $1,n,2,3,n-1,1$ &  $x\leq z\leq 1\leq y\leq a$  & $1,n,3,n-1,2,1$ \\ \hline   & & & \\
$1\leq x\leq y,z \leq a$ & $1,3,n,n-1,2,1$ & $x\leq a\leq y\leq 1\leq z$  & $1,n,2,n-1,3,1$  \\ \hline   & & & \\
$x\leq a\leq z\leq 1\leq y$  & $1,n,3,2,n-1,1$ & $x\leq y\leq 1\leq z\leq a$  & $1,n,n-1,2,3,1$ \\ \hline   & & & \\
$x\leq y,z\leq a\leq 1$  & $1,n,n-1,3,2,1$  & &
 \\ \hline
   \end{tabular}
   $\ \ \ \ \ \ \ \ $%
\caption{Efficiency of the Perron vector of $Z_{n}(x,y,z,a)$ with a directed Hamiltonian cycle.  }\label{tab1}%
%TCIMACRO{\TeXButton{E}{\end{table}}}%
%BeginExpansion
\end{table}%

In Table \ref{tab2}, we consider the cases that, using Lemmas \ref{le-1} and \ref{le-2}, produce two directed cycles whose union give us the strong connectivity of the digraph $G_{Z_n(x,y,z,a),w}$. If $b,c$ are two real numbers, we write $b,c$ to denote $\min\{b,c\}\leq \max\{b,c\}$.

\begin{table}[H] \centering
\footnotesize \begin{tabular}{|c|c||c|c|} \hline & & & \\ relation among $1,x,y,z,a$ &  cycle & relation among $1,x,y,z,a$ &  cycle \\ \hline\hline  & & & \\
 $x\leq 1\leq y, z\leq a$  & $2,n-1,3,2$ and & $x\leq a\leq y, z\leq 1$  & $3,1,n,3$ and \\
  &  $1,n,n-1,2,1$ & & $3,2,n-1,3$ \\\hline  & & & \\
  $x\leq z,y\leq 1\leq a$  & $1,n,3,1$ and & $1\leq x\leq a\leq y,z$ & $3,n,2,3$ and\\
&  $1,n,n-1,2,1$ & & $3,n-1,1,3$ \\\hline
\end{tabular}
   $\ \ \ \ \ \ \ \ $%
\caption{Efficiency of the Perron vector of $Z_{n}(x,y,z,a)$ with two directed cycles.  }\label{tab2}%
%TCIMACRO{\TeXButton{E}{\end{table}}}%
%BeginExpansion
\end{table}%

In Table \ref{tab3}, we consider the cases that, using Lemmas \ref{le-1}, \ref{le-2} and Theorem \ref{t1}, produce a directed cycle and an extra edge whose union give us the strong connectivity of the digraph $G_{Z_n(x,y,z,a),w}$. Recall that by Theorem \ref{t1} we know there is no source in the digraph $G_{Z_n(x,y,z,a),w}$.

\begin{table}[H] \centering
\footnotesize \begin{tabular}{|c|c|c|c|} \hline & & & \\ relation among $1,x,y,z,a$ &  cycle & extra edge & possible source \\ \hline\hline  & & & \\
$1\leq x\leq z\leq a\leq y$ & $3,n,n-1,1,3$ &  $(2,3)$  & $2$ \\ \hline  & &  &\\
$x\leq y\leq a\leq z\leq 1$ & $3,2,1,n,3$ &  $(n-1,3)$ & $n-1$  \\ \hline
\end{tabular}
   $\ \ \ \ \ \ \ \ $%
\caption{Efficiency of the Perron vector of $Z_{n}(x,y,z,a)$ with a directed cycle and an extra edge.  }\label{tab3}%
%TCIMACRO{\TeXButton{E}{\end{table}}}%
%BeginExpansion
\end{table}%

In Table \ref{tab4}, we consider the cases that, using Lemmas \ref{le-1}, \ref{le-2} and Theorem \ref{t1}, produce a directed cycle and two extra edges whose union give us the strong connectivity of the digraph $G_{Z_n(x,y,z,a),w}$. Recall that by Theorem \ref{t1} we know there is no source in the digraph $G_{Z_n(x,y,z,a),w}$.

\begin{table}[H] \centering
\footnotesize \begin{tabular}{|c|c|c|c|} \hline & & & \\ relation among $1,x,y,z,a$ &  cycle & two extra edges & possible source \\ \hline\hline  & & &  \\
$x\leq 1\leq z\leq a\leq y$ & $1,n,n-1,1$ &  $(2,3),\ (3,n-1)$ & $2$  \\ \hline  & & &  \\
$x\leq y\leq a\leq 1\leq z$ & $1,n,2,1$ &  $(n-1,3),\ (3,1)$ & $n-1$  \\ \hline
\end{tabular}
   $\ \ \ \ \ \ \ \ $%
\caption{Efficiency of the Perron vector of $Z_{n}(x,y,z,a)$ with a directed cycle and two extra edges.  }\label{tab4}%
%TCIMACRO{\TeXButton{E}{\end{table}}}%
%BeginExpansion
\end{table}%

Using Tables \ref{tab1}, \ref{tab2}, \ref{tab3}, \ref{tab4} we get the result.
 \end{proof5}

\vspace{2ex}

\begin{lemma}\label{le-3} Let $n\geq 5$ and $Z_n(x,y,z,a)$, with $x,y,z,a>0$,  be the matrix in (\ref{recmatrix}). Let $w$ be the Perron eigenvector of $Z_n(x,y,z,a)$.
\begin{enumerate}
\item If $(3,2)$ is an edge of $G_{Z_n(x,y,z,a),w}$, $\max\{a,z\}\leq 1$ and $a\neq z$, then $(2,3)$ is not an edge of $Z_n(x,y,z,a)$.
    \item If $(3,1)$ is an edge of $G_{Z_n(x,y,z,a),w}$, $\max\{x,y\}\leq 1$ and $x\neq y$, then $(1,3)$ is not an edge of $Z_n(x,y,z,a)$.
        \item If $(3,n)$ is an edge of $G_{Z_n(x,y,z,a),w}$, $\min\{x,z\}\geq 1$ and $x\neq z$, then $(n,3)$ is not an edge of $Z_n(x,y,z,a)$.
             \item If $(3,n-1)$ is an edge of $G_{Z_n(x,y,z,a),w}$, $\min\{y,a\}\geq 1$ and $a\neq y$, then $(n-1,3)$ is not an edge of $Z_n(x,y,z,a)$.
    \end{enumerate}
\end{lemma}

\begin{proof} \begin{enumerate}
\item Suppose that $(2,3)$ is an edge of $Z_n(x,y,z,a)$. Then $\frac{w_2}{w_3}=1$.
By equation \ref{j17}, we get $$(a-1)w_{n-1}+(z-1)w_{n}=0.$$ This is impossible because $\max\{a,z\}\leq 1$ and $a\neq z$.
\item Suppose that $(1,3)$ is an edge of $Z_n(x,y,z,a)$. Then $\frac{w_1}{w_3}=1$.
By equation \ref{j14}, we get $$(y-1)w_{n-1}+(x-1)w_{n}=0.$$ This is impossible because $\max\{x,y\}\leq 1$ and $x\neq y$.
\item Suppose that $(n,3)$ is an edge of $Z_n(x,y,z,a)$. Then $\frac{w_n}{w_3}=1$.
By equation \ref{j21}, we get $$\left(1-\frac{1}{x}\right)w_{1}+\left(1-\frac{1}{z}\right)w_{2}=0.$$ This is impossible because $\min\{x,z\}\geq 1$ and $x\neq z$.
\item Suppose that $(n-1,3)$ is an edge of $Z_n(x,y,z,a)$. Then $\frac{w_{n-1}}{w_3}=1$.
By equation \ref{j20}, we get $$\left(1-\frac{1}{y}\right)w_{1}+\left(1-\frac{1}{a}\right)w_{2}=0.$$ This is impossible because $\min\{a,y\}\geq 1$ and $a\neq y$.
\end{enumerate}
\end{proof}

\vspace{2ex}

\begin{proof6}  Let
\[
w=\left[
\begin{array}
[c]{ccccc}%
w_{1} & w_{2} &\cdots & w_{n-1} & w_{n}%
\end{array}
\right]  ^{T}
\]
be the Perron vector associated with the Perron eigenvalue $r$ of $Z_n(x,y,z,a)$, with $n\geq 5$. Recall that if $n>5$ then $w_i=w_3$, for $4\leq i\leq n-2$, and  $3\rightarrow i$ and $i\rightarrow 3$ are edges of $G_{Z_n(x,y,z,a),w}$.

Suppose that $G_{Z_n(x,y,z,a),w}$ has a sink. By Theorem \ref{blanq}, $w$ is inefficient for $Z_n(x,y,z,a).$

Conversely, suppose that $w$ is inefficient for $Z_n(x,y,z,a).$ Using the statements of the Theorems \ref{main3}, \ref{main4}, \ref{main5}, \ref{main6}, we conclude that there are $32$ cases to study.

To facilitate the proof and using Lemmas \ref{le-1} and \ref{le-2}, in the following tables we show when the Perron eigenvector is inefficient.

In Table \ref{tab5}, we consider the cases that, using Lemmas \ref{le-1}, \ref{le-2}, produce a directed cycle and an extra edge.

\begin{table}[H] \centering
\footnotesize \begin{tabular}{|c|c|c|c|} \hline & & & \\ relation among $1,x,y,z,a$ &  cycle & extra edge & possible sink \\ \hline\hline  & & & \\
$z< x< y< a\leq 1$ & $3,2,n,n-1,3$ &  $(3,1)$  & $1$ \\ \hline  & & & \\
$x< z< a< y\leq 1$ & $3,1,n,n-1,3$ &  $(3,2)$  & $2$ \\ \hline  & & & \\
$y< a< z< x\leq 1$ & $3,1,n-1,n,3$ &  $(3,2)$ & $2$  \\ \hline  & &  &\\
$a< y< x< z\leq 1$ & $3,2,n-1,n,3$ &  $(3,1)$  & $1$ \\ \hline  & &  &\\
$1\leq a< z< x< y$ & $3,n-1,1,2,3$ &  $(3,n)$  & $n$ \\ \hline  & & & \\
$1\leq y< x< z< a$ & $3,n-1,2,1,3$ &  $(3,n)$ & $n$  \\ \hline  & & & \\
$1\leq z< a< y< x$ & $3,n,1,2,3$ &  $(3,n-1)$ & $n-1$  \\ \hline  & & &  \\
$1\leq x< y< a< z$ & $3,n,2,1,3$ &  $(3,n-1,3)$ & $n-1$ \\\hline
\end{tabular}
   $\ \ \ \ \ \ \ \ $%
\caption{Inefficiency of the Perron vector of $Z_{n}(x,y,z,a)$ with a directed cycle and an extra edge.  }\label{tab5}%
%TCIMACRO{\TeXButton{E}{\end{table}}}%
%BeginExpansion
\end{table}%

In Table \ref{tab6}, we consider the cases that, using Lemmas \ref{le-1}, \ref{le-2}, produce a directed cycle and two extra edges. Using Lemma \ref{le-3}, we conclude that if $w$ is not efficient for $Z_n(x,y,z,a)$ then  $G_{Z_n(x,y,z,a),w}$ has a sink.

\begin{table}[H] \centering
\footnotesize \begin{tabular}{|c|c|c|c|} \hline & & & \\ relation among $1,x,y,z,a$ &  cycle & two extra edges & possible sink \\ \hline\hline  & & &  \\
$x< z< a< 1\leq y$ & $1,n,n-1,1$ &  $(3,2),\ (n,3)$ & $2$  \\ \hline  & & & \\
$a< 1\leq z< x< y$ & $1,2,n-1,1$ &  $(3,n),\ (1,3)$ & $n$  \\ \hline  & & & \\
$z<x< y\leq 1< a$ & $2,n,n-1,2$ &  $(3,1),\ (n,3)$  & $1$ \\ \hline  & & & \\
$y< 1\leq x< z< a$ & $1,n-1,2,1$ &  $(3,n),\ (2,3)$ & $n$  \\ \hline  & & & \\
$y< a< z< 1\leq x$ & $1,n-1,n,1$ &  $(3,2),\ (n-1,3)$ & $2$  \\ \hline  & & &  \\
$z< 1\leq a< y< x$ & $1,2,n,1$ &  $(3,n-1),\ (1,3)$ & $n-1$  \\ \hline  & & & \\
$a< y< x< 1\leq z$ & $2,n-1,n,2$ &  $(3,1),\ (n-1,3)$ & $1$  \\ \hline  & & & \\
$x< 1\leq y< a< z$ & $1,n,2,1$ &  $(3,n-1),\ (2,3)$ & $n-1$  \\\hline
\end{tabular}
   $\ \ \ \ \ \ \ \ $%
\caption{Inefficiency of the Perron vector of $Z_{n}(x,y,z,a)$ with a directed cycle and two extra edges.  }\label{tab6}%
%TCIMACRO{\TeXButton{E}{\end{table}}}%
%BeginExpansion
\end{table}%

In Table \ref{tab7}, we consider the cases that, using Lemmas \ref{le-1} and \ref{le-2}, produce a directed cycle without a vertex (possible sink).  Note that the relations among $1,x,y,z,a$, except the boundary, and the cycle that arise in the first two columns of Table \ref{tab7} are the reverses of the relations among $1,x,y,z,a$ and the cycle that arise in the third and fourth columns of Table \ref{tab7}. If $b,c$ are two real numbers, we write $b,c$ to denote $\min\{b,c\}\leq \max\{b,c\}$.

\begin{table}[H] \centering
\footnotesize \begin{tabular}{|c|c||c||c|c|} \hline & && & \\ relation among $1,x,y,z,a$ &  cycle & possible sink & relation among $1,x,y,z,a$ &  cycle \\ \hline\hline  & && & \\
$x<z<1\leq a<y$ & $1,n,3,n-1,1$ &2&  $y<a<1\leq z<x$  & $1,n-1,3,n,1$ \\ \hline  & && & \\
$z<a\leq 1<y<x$ & $1,3,2,n,1$ &n-1 &  $x<y\leq 1<a<z$  & $1,n,2,3,1$ \\
\hline  & & && \\
$z,y<1<a,x$ & $1,n-1,2,n,1$ & 3& $ x, a< 1<z, y$  & $1,n,2,n-1,1$  \\
 \hline  & & && \\
$y<x\leq 1<z<a$ & $1,n-1,2,3,1$ &n&  $a<z\leq 1<x<y$  & $1,3,2,n-1,1$  \\
\hline  & & && \\
$a<y<1\leq x<z$ & $3,n,2,n-1,3$  &1&  $z<x<1\leq y<a$  & $3,n-1,2,n,3$  \\\hline
\end{tabular}
   $\ \ \ \ \ \ \ \ $%
\caption{Inefficiency of the Perron vector of $Z_{n}(x,y,z,a)$ with a directed cycle without a vertex.  }\label{tab7}%
%TCIMACRO{\TeXButton{E}{\end{table}}}%
%BeginExpansion
\end{table}%

Using tables \ref{tab5}, \ref{tab6} and \ref{tab7} we get the result.
\end{proof6}

\section{Conclusions\label{scon}}

\hspace{3ex}In this paper we focus on the extensions of a ${\cal PC}_n$ matrix. More specially, in a conjecture proposed by Furtado and Johnson in \cite{f2} concerning the non-existence of a source in the digraph associated with an extension of a ${\cal PC}_n$ matrix and its Perron eigenvector. Using this conjecture, we studied the efficiency of the Perron eigenvector for some
matrices obtained from
consistent matrices by perturbing four entries and the corresponding
reciprocal entries, so that there is a $2\times2$ submatrix containing four
of the perturbed entries and not containing a diagonal entry.  Moreover, we showed that the inefficiency of the Perron eigenvector of our studied matrices is equivalent to the existence of a sink in the digraph associated with the matrix and the Perron eigenvector.

We also commented on known results and the change in the relations of the weights of the corresponding components of the Perron eigenvectors of a ${\cal PC}_n$ matrix, $A$, of an extension of $A$. So, this work leaves a question when a ${\cal PC}_n$ matrix $A$ has an inefficient Perron eigenvector.

Is there an extension $B$ of $A$, with efficient Perron eigenvector, having the same relations of the weights of the components of the Perron eigenvector of $A$ and the first $n$ components of the Perron eigenvector of $B$?


\begin{thebibliography}{99}                                                                                               %


\bibitem {p6}K. \'{A}bele-Nagy and S.  Boz\'{o}ki.
 Efficiency analysis of simple perturbed pairwise comparison matrices.
\textit{Fundamenta Informaticae}, 144(3-4) (2016) 279-289.

\bibitem {p2}K. \'{A}bele-Nagy, S.  Boz\'{o}ki. and \"{O} Reb\'{a}k.
  Efficiency analysis of double perturbed pairwise
comparison matrices. Journal of the Operational Research Society, 69(5) (2018) 707-713.




\bibitem {blanq2006}R. Blanquero, E. Carrizosa and E. Conde. Inferring
efficient weights from pairwise comparison matrices. \textit{Mathematical
Methods of Operations Research}, 64(2) (2006) 271-284.




\bibitem {european}S. Boz\'{o}ki and J. F\"{u}l\"{o}p. Efficient weight
vectors from pairwise comparison matrices. \textit{European Journal of
Operational Research}, 264(2)  (2018) 419-427.













\bibitem {Fer}R. Fernandes and S. Furtado.  Efficiency of the principal eigenvector of some triple perturbed consistent matrices, \textit{European Journal of Operational Research,} 298(3) (2022) 1007-1015.

\bibitem {Fer2}R. Fernandes and S. Palheira. Triple perturbed consistent matrices and the efficiency of its principal right eigenvector, \textit{International Journal of Approximate Reasoning}, 170 (2024) 109204.

\bibitem {Rub}R. Fernandes and R. Palma. Positive vectors, pairwise comparison matrices and
directed Hamiltonian cycles, \textit{Linear Algebra and its Applications}, 699 (2024) 312-330.

\bibitem {f2}S. Furtado and C.R. Johnson.  Efficiency analysis for the Perron vector of a reciprocal matrix, \textit{Applied Mathematics and Computation}, 480 (2024) 128913.

    \bibitem {f1}S. Furtado and C.R. Johnson.  Efficient vectors for block perturbed consistent matrices, \textit{SIAM Journal on Matrix Analysis and Applications}, 45 (2024) 601-618.

        \bibitem {f4}S. Furtado and C.R. Johnson.  Efficiency of the convex hull of the columns of certain triple perturbed consistent matrices analysis, \textit{Advances in Operator Theory}, 9 (2024) 89.

\bibitem {HJ}R.A. Horn and C.R. Johnson. \textit{Matrix analysis}.
Cambridge University Press, Cambridge (1985).





\bibitem {saaty1977}T.L. Saaty. \textit{A scaling method for
priorities in hierarchical structures}. \textit{Journal of Mathematical
Psychology,} 32(3)  (1977) 234-281.

\bibitem {Saaty} T.L. Saaty. The Analytic Hierarchy Process,
McGraw-Hill, New York,  (1980).


\bibitem {thakkar2021multi}J.J. Thakkar. Multi-criteria decision making. Springer, vol 336 (2021).

\end{thebibliography}
\end{document}